\documentclass{elsarticle}
\makeatletter
\def\ps@pprintTitle{%
 \let\@oddhead\@empty
 \let\@evenhead\@empty
 \def\@oddfoot{}%
 \let\@evenfoot\@oddfoot}
\makeatother
\usepackage[english]{babel}
\usepackage{amsfonts}
\usepackage{xcolor}
\usepackage{amsthm}
\usepackage{latexsym}
\usepackage{amsmath}
\usepackage{amssymb}
\usepackage{array}
\usepackage{arydshln}

\usepackage{enumerate}
\usepackage[a4paper]{geometry}
\usepackage{algorithm}
\usepackage[noend]{algpseudocode}

\newtheorem{theorem}{Theorem}

\newtheorem{example}{Example}

\newtheorem{corollary}{Corollary}
\newtheorem{lemma}{Lemma}
\newtheorem{claim}{Claim}
\newtheorem{proposition}{Proposition}
\newtheorem{question}{Question}
\numberwithin{equation}{section}
\newtheorem{conjecture}{Conjecture}
\theoremstyle{definition}
\newtheorem{definition}{Definition}

\begin{document}
\title{On the additive complexity of a Thue-Morse like sequence}
\author[hzau]{Jin Chen}
\address[hzau]{College of Science, Huazhong Agricultural University, Wuhan 430070, China.}
\ead{cj@mail.hzau.edu.com}
\author[hust]{Zhixiong Wen}
\address[hust]{School of Mathematics and Statistics, Huazhong University of Science and Technology, Wuhan, 430074, China.}
\ead{zhi-xiong.wen@hust.edu.cn}
\author[scut]{Wen Wu\corref{corres}}
\cortext[corres]{Wen Wu is the corresponding author.}
\address[scut]{School of Mathematics, South China University of Technology, Guangzhou 510641, China.}
\ead{wuwen@scut.edu.cn}

\begin{abstract}
In this paper, we study the additive complexity $\rho^{+}_{\mathbf{t}}(n)$ of a Thue-Morse like sequence $\mathbf{t}=\sigma^{\infty}(0)$ with the morphism $\sigma: 0\to 01, 1\to 12, 2\to 20$. We show that $\rho^{+}_{\mathbf{t}}(n)=2\lfloor\log_2(n)\rfloor+3$ for all integers $n\geq 1$. Consequently,  $(\rho_{\mathbf{t}}(n))_{n\geq 1}$ is a $2$-regular sequence.
\end{abstract}

\begin{keyword}
Thue-Morse like sequence\sep   Additive complexity\sep $k$-regular sequence
\MSC[2010]{05D99\sep 11B85}
\end{keyword}

\maketitle

\section{Introduction}
{Recently the study of the abelian complexity of infinite words was initiated by G. Richomme, K. Saari, and L. Q. Zamboni \cite{RSZ11}. For example, the abelian complexity functions of some notable sequences, such as the Thue-Morse sequence and all Sturmian sequences, were studied in \cite{RSZ11} and \cite{CH73} respectively. There are also many other works including the unbounded abelian complexity, see \cite{BBT11,LCWW17,LCW18,MR13,RSZ10} and references therein.  At the mean time, many authors had devoted to the generalizations of the abelian complexity. For instance, $l$-abelian complexity, cyclic complexity and binomial complexity are first presented in \cite{KSZ13}, \cite{CFSZ14} and \cite{RS15} respectively. In 1994, G. Pirillo and S. Varricchio \cite{PV94} raised the following question: do there exist infinite words avoiding additive squares or additive cubes? Based on this infamous problem, H. Ardal, T. Brown, V. Jungi\'{c} and J. Sahasrabudhe proposed the additive complexity for infinite word on a finite subset of $\mathbb{Z}$ in \cite{ABJS12}.  It follows from the definition of additive equivalence in Section $2$ that the additive complexity $\{\rho^{+}(n)\}$  coincides with the abelian complexity $\{\rho^{ab}(n)\}$ for every infinite word on the alphabet composed of two integers. For every infinite word on the alphabet composed of integers whose cardinality is not less than three, it is easy to know that $\rho^{+}(n) \leq \rho^{ab}(n)$ for every $n.$ }

Let $\sigma$ be the morphism $0\mapsto 01, 1\mapsto 12, 2 \mapsto 20$ on $\{0,1,2\}$ and $\mathbf{t}:=\sigma^{\infty}(0)$. The infinite sequence $\mathbf{t}$ is a Thue-Morse like sequence (see \cite{AS03,Sloane}). Further, $\mathbf{t}$ is $2$-automatic  and uniformly recurrent (see \cite{F02}). A sequence $\mathbf{w}=w_0w_1w_2\cdots$ is a \emph{$k$-automatic} sequence if its \emph{$k$-kernel}
$\{ (w_{{k^e}n + c})_{n \ge 0}~|~ {e \ge 0,0 \le c < k^e}\} $
finite. If the $\mathbb{Z}$-module generated by its $k$-kernel is
finitely generated, then $\mathbf{w}$ is a \emph{$k$-regular} sequence.

%For some details about $\mathbf{t}$, please see \cite{Sloane}.

In this paper, we investigate the additive complexity function
$\rho^{+}_{\mathbf{t}}(n)$ of $\mathbf{t}$, where $\rho^{+}_{\mathbf{t}}(n)$ is the number of different digit sums of all words (of length $n$) that occur in $\mathbf{t}$. We give the explicit value of $(\rho^{+}_{\mathbf{t}}(n))_{n\geq 1}$. 
\begin{theorem}\label{thm:main}
For all integer $n\geq 1$, \[\rho_{\mathbf{t}}^{+}(n)  = 2\lfloor \log_2 n\rfloor+3.\] 
\end{theorem}
Consequently, we know that the additive complexity function $(\rho_{\mathbf{t}}(n))_{n\geq 1}$ satisfying the recurrence relations: $\rho_{\mathbf{t}}^{+}(1)=3$ and for all $n\geq 1$,
\[\rho_{\mathbf{t}}^{+}(2n) =  \rho_{\mathbf{t}}^{+}(2n+1) =  \rho_{\mathbf{t}}^{+}(n)+2.
\]
The above recurrence relations imply the regularity of $(\rho_{\mathbf{t}}(n))_{n\geq 1}$. 
\begin{corollary}
The additive complexity $(\rho_{\mathbf{t}}(n))_{n\geq 1}$ of $\mathbf{t}$ is a $2$-regular sequence. 
\end{corollary}

From the above corollary, it is natural to have the following conjecture. 
\begin{conjecture}
The additive complexity of any $k$-automatic sequence is a $k$-regular sequence.
\end{conjecture}

This paper is organized as follows. In Section 2, we give some notations. In Section 3, we prove Theorem \ref{thm:main}. The proof is separated into $3$ steps. Each step gives a more specific result. 

\section{Preliminaries}
An \emph{alphabet} $\mathcal{A}$ is a finite and non-empty set (of symbols) whose elements are called \emph{letters}. A (finite)
\emph{word} over the  alphabet $\mathcal{A}$ is a concatenation of letters in $\mathcal{A}$. The concatenation of two words ${u} =
u_{0}u_{1} \cdots u(m)$ and ${v} = v_{0}v_{1} \cdots v_{n}$ is the word ${uv} = u_{0}u_{1} \cdots u_{m}v_{0}v_{1} \cdots v_{n}$. The set of all finite
words over $\mathcal{A}$ including the \emph{empty word} $\varepsilon $ is denoted by $\mathcal{A}^*$. An infinite word $\mathbf{w}$ is an
infinite sequence of letters in $\mathcal{A}$. The set of all infinite words over $\mathcal{A}$ is denoted by $\mathcal{A}^{\mathbb{N}}$.

The \emph{length} of a finite word ${w}\in \mathcal{A^*}$, denoted by $|w|$, is the number of letters contained in $w$. We set $\left|
\varepsilon  \right| = 0$. For any word $u\in\mathcal{A}^{*}$ and any letter $a \in \mathcal{A}$, let $|{u}|_a$ denote the number of occurrences
of $a$ in ${u}$.

A word $w$ is a factor of a finite (or an infinite) word $v$, written by $w\prec v$ if there exist a finite word $x$ and a finite (or an
infinite) word $y$ such that $v=xwy$. When $x=\varepsilon$,  ${w}$ is called a \emph{prefix} of ${v}$, denoted by ${w} \triangleleft {v}$; when
$y=\varepsilon$, $ w$ is called a suffix of ${v}$, denoted by ${w} \triangleright {v}$.

For a real number $x$, let $\lfloor x \rfloor$ (resp. $\lceil x \rceil$)  be the integer that is less (resp. larger) than  or equal to  $x.$  For every natural number $n$ and some positive integer $b\geq 2,$ set $(n)_b$ be the regular $b$-ary expansion of $n$.

\subsection{Additive complexity}
Now we assume that $\mathcal{A}\subset\mathbb{Z}$. Let \[\mathbf{w}=w_{0}w_{1}w_{2}\cdots \in \mathcal{A}^{\mathbb{N}}\] be an
infinite word. Denote by ${\mathcal{F}_{\mathbf{w}}(n)}$ the set of all
factors of $\mathbf{w}$ of length $n$, i.e., \[{\mathcal{F}_{\mathbf{w}}}(n): = \{w_{i}w_{i + 1}\cdots w_{i + n - 1} : i \geq 0 \}.\]
Write $\mathcal{F}_{\mathbf{w}}=\cup_{n\geq 1}\mathcal{F}_{\mathbf{w}}(n)$.
The \emph{subword complexity function} ${\rho_{\mathbf{w}}}:~\mathbb{N} \to \mathbb{N}$ of $\mathbf{w}$ is defined by
\[\rho_{\mathbf{w}}(n) := \# \mathcal{F}_{\mathbf{w}}(n).\]
Denote the \emph{digit sum} of ${u}=u_0\cdots u_{|{u}|-1}\in\mathcal{A}^{*}$ by
\[\mathrm{DS}(u):= \sum_{j=0}^{|{u}|-1} u_{j}.\]
Two finite words $u$, $v\in\mathcal{A}^*$ is \emph{additive equivalent} if $\text{DS}(u)=\text{DS}(v)$. The additive equivalent induces an equivalent relation, denoted by  $\sim_{+}$. 
\begin{definition}
The \emph{additive subword complexity function} ${\rho_{\mathbf{w}}}^{+}:~\mathbb{N} \to \mathbb{N}$ of $\mathbf{w}$ is defined by
\[\rho_{\mathbf{w}}^{+}(n) := \# \{\mathcal{F}_{\mathbf{w}}(n)/\sim_{+}\}.\]
\end{definition}
In fact, 
\begin{equation}\label{eq:rhos}
\rho_{\mathbf{w}}^{+}(n) = \#\{\text{DS}(u): u \in  \mathcal{F}_{\mathbf{w}}(n)\}.
\end{equation}

The additive complexity can also be obtained from the Parikh vector. Let $\mathbf{u}$ be an infinite word on an alphabet $\mathcal{A}= \{a_0, a_1, \cdots, a_{q-1}\}$ and $v, $ a
factor of $\mathbf{u}$. The Parikh vector of $v$ is the $q$-uplet \[\psi(w) = (|v|_{a_0} , |v|_{a_1} , \cdots, |v|_{a_{q-1}} ).\] Denote by $\Psi_{\mathbf{u}}(n)$, the set of the Parikh vectors of the factors of length $n$ of ${\mathbf{u}}$:
\[ \Psi_{\mathbf{u}}(n) = \{\psi(v) : v \in \mathcal{F}_{\mathbf{u}}(n)\}. \]

The abelian complexity of ${\mathbf{u}}$ is defined by: $\rho^{ab}_{\mathbf{u}}(n) = \#\Psi_{\mathbf{u}}(n)$. Given any infinite word $\mathbf{w}$ on an alphabet $\mathcal{B}= \{b_0, b_1, \cdots, b_{q-1}\} \in \mathbb{Z}^q$ with $b_0<b_1<\cdots <b_{q-1}$.  It is easy to verify that the additive complexity of $\mathbf{w}$ is 
\begin{equation}\label{eq:rhos2} \rho_{\mathbf{w}}^{+}(n) = \#\left\{<(b_0,b_1, \cdots, b_{q-1}),  \psi(v)> : v \in \mathcal{F}_{\mathbf{w}}(n)\right\}
\end{equation}
where $<\cdot,\cdot>$ is the usual inner product in the Euclid space $\mathbb{R}^{q}.$ In fact, \[\text{DS}(u) = <(b_1, \cdots, b_{q-1}),  \psi(u)>\] for every factor $u \in \mathcal{F}_{\mathbf{w}}(n)$.

\section{Additive complexity of $\mathbf{t}$}
In this section, we prove Theorem \ref{thm:main}. According to \eqref{eq:rhos}, the study of the additive complexity function turns out to be the study of digit sums of all factors. Our strategy in the proof of Theorem \ref{thm:main} is the following:
\begin{itemize}
\item (Proposition \ref{prop:dsubounds})
give the upper and lower bounds of $\rho^{+}_{\mathbf{t}}(n)$ for all $n\geq 1$;
\item (Proposition \ref{lem:max_f})
show that the upper and lower bounds can be attained;
\item (Proposition \ref{prop:3})
study the all the accessible values of digit sums.  
\end{itemize}
Then, Theorem \ref{thm:main} follows from Proposition \ref{prop:dsubounds}, \ref{lem:max_f} and \ref{prop:3}.
\subsection{Upper and lower bounds of digit sums of factors}
\begin{proposition}\label{prop:dsubounds}
For every integer $n\geq 1$, 
\[n-\lfloor \log_2 n\rfloor-1\leq \mathrm{DS}(u)\leq n+\lfloor \log_2 n\rfloor+1\]
for all $u\in\mathcal{F}_{\mathbf{t}}(n)$.
\end{proposition}
Note that for every $u=u_0u_1\cdots u_{n-1} \in \mathcal{F}_{\mathbf{t}}(n)$, 
\begin{align}
\mathrm{DS}(u) & = \sum_{i=0}^{n-1}u_i=0\cdot |u|_0+1\cdot |u|_1+2\cdot |u|_2 =|u|_0+|u|_1+2|u|_2-|u|_0\notag\\
& = n+|u|_2-|u|_0.\label{eq:dsn20}
\end{align}
To prove Proposition \ref{prop:dsubounds}, we only need to show that for all $u\in\mathcal{F}_{\mathbf{t}}(n)$, 
\begin{equation}\label{eq:ulbounds}
-\lfloor \log_2 n\rfloor-1\leq |u|_2-|u|_0\leq \lfloor \log_2 n\rfloor+1.
\end{equation}

The following lemmas are aimed to analysis the quantity $|u|_2-|u|_0$.
\begin{lemma}\label{lem:image_num}
For every $u \in \{0,1,2\}^{*}$,  
\begin{eqnarray*}
|\sigma(u)|_2-|\sigma(u)|_0 &=& |u|_1-|u|_0, \\
|\sigma(u)|_1-|\sigma(u)|_0 &=& |u|_1-|u|_2.
\end{eqnarray*}
\end{lemma}
\begin{proof}
It follows from the definition of $\sigma$ that 
\[|\sigma(u)|_0 = |u|_0+|u|_2,\quad |\sigma(u)|_1 = |u|_0+|u|_1,\quad |\sigma(u)|_2 = |u|_1+|u|_2.\]
The above equations give the required results.
\end{proof}

Let $a,b,c$ be any arrangement of $0,1,2.$ Define $\tau_c : a\mapsto b, b \mapsto a, c \mapsto c$. For every finite word $w=w_0w_1\cdots w_{n-1}w_n \in\{0,1,2\}^*$, let $w^R=w_{n}w_{n-1}\cdots w_{1}w_{0}$ be the mirror of $w$. For every $x\in\{0,1,2\}$, write $\underline{x}:=x-1\, (\mathrm{mod}~3)$ and $\overline{x}:=x+1\, (\mathrm{mod}~3)$. The morphisms $\sigma$ and $\tau$ have the following commutative property.

\begin{lemma}\label{lem:sigmatau}
For every $u\in\mathcal{F}_{\mathbf{t}}$ and every $c=0,1,2$,
\begin{equation}\label{eqn:factor2}
\sigma( {\tau_{c}(u)}^R) =  {\tau_{\underline{c}}(\sigma(u))}^R. 
\end{equation}
\end{lemma}
\begin{proof}
It is easy to check \eqref{eqn:factor2} for all $u\in\mathcal{F}_{\mathbf{t}}(1)=\{0,1,2\}$. 
Assume that \eqref{eqn:factor2} holds for all $u\in\cup_{i=1}^{n-1}\mathcal{F}_{\mathbf{t}}(i)$. For any $u\in\mathcal{F}_{\mathbf{t}}(n)$, we have $u=va$ where $v\in\mathcal{F}_{\mathbf{t}}(n-1)$ and $a\in\{0,1,2\}$. Then 
\begin{align*}
 \sigma( {\tau_{c}(u)}^R) &= \sigma( {\tau_{c}(va)}^R) =  \sigma( \tau_{c}(a)^R\tau_c(v)^R)\notag\\
				      &= \sigma( {\tau_{c}(a)}^R) \sigma({\tau_{c}(v)}^R) =  {\tau_{\bar{c}}(\sigma(a))}^R {\tau_{\bar{c}}(\sigma(v))}^R ~~~~~~{\text{(by the assumption)}} \\
				      &= {\tau_{\bar{c}}(\sigma(va))}^R = {\tau_{\bar{c}}(\sigma(u))}^R,\notag
\end{align*}
which implies that \eqref{eqn:factor2} holds for all $u\in\mathcal{F}_{\mathbf{t}}(n)$ and $c=0,1,2$.
\end{proof}

While $\sigma$ maps every factor of $\mathbf{t}$ to a factor of $\mathbf{t}$, the morphism $\tau_c$ maps every factor of $\mathbf{t}$ to the mirror of some factor of $\mathbf{t}$.
\begin{lemma}\label{lem:fab}
If $u\in\mathcal{F}_{\mathbf{t}}$, then $\tau_c(u)^R\in\mathcal{F}_{\mathbf{t}}$ for $c=0,1,2$.
\end{lemma}
\begin{proof}
When $u\in\mathcal{F}_{\mathbf{t}}(1)\cup\mathcal{F}_{\mathbf{t}}(2)$, the result can be checked directly. Now, suppose the result holds for all  $u\in\cup_{i=1}^{n-1}\mathcal{F}_{\mathbf{t}}(i)$ (where $n\geq 3)$.
Let $u\in\mathcal{F}_{\mathbf{t}}(n)$. If $n$ is odd, then $u=a\sigma(v)$ or $\sigma(v)b$ where $v\in\mathcal{F}_{\mathbf{t}}(\lfloor n/2\rfloor)$ and $a,b\in\{0,1,2\}$, which also imply that $\underline{a}u=\sigma(\underline{a}v)$ or $u\overline{b}=\sigma(vb)$ with $\underline{a}v,vb\in\mathcal{F}_{\mathbf{t}}(\frac{n+1}{2})$. By Lemma \ref{lem:sigmatau}, for $c=0,1,2$,
\[\tau_c(\underline{a}u)^R  =\tau_c(\sigma(\underline{a}v))^R = \sigma(\tau_{\overline{c}}(\underline{a}v)^R).\]
Since $\underline{a}v\in\mathcal{F}_{\mathbf{t}}(\frac{n+1}{2})$, by the inductive hypothesis, $\tau_{\overline{c}}(\underline{a}v)^R\in\mathcal{F}_{\mathbf{t}}(\frac{n+1}{2})$. So $\tau_c(\underline{a}u)^R\in\mathcal{F}_{\mathbf{t}}(n+1)$ and $\tau_c(u)^R\in\mathcal{F}_{\mathbf{t}}(n)$. The same is true for the case $u=\sigma(v)b$.

If $n$ is even, then $u=\sigma(w)$ or $a\sigma(v)b$ where $w\in\mathcal{F}_{\mathbf{t}}(n/2)$, $v\in\mathcal{F}_{\mathbf{t}}(\frac{n}{2}-1)$ and $a,b\in\{0,1,2\}$. When $u=a\sigma(v)b$, we have \(\underline{a}ub=\sigma(\underline{a}vb)\) with $\underline{a}vb\in\mathcal{F}_{\mathbf{t}}(\frac{n}{2}+1)$. By Lemma \ref{lem:sigmatau}, for $c=0,1,2$, \[\tau_c(\underline{a}ub)^R  =\tau_c(\sigma(\underline{a}vb))^R = \sigma(\tau_{\overline{c}}(\underline{a}vb)^R).\]
By the inductive hypothesis, $\tau_{\overline{c}}(\underline{a}vb)^R\in\mathcal{F}_{\mathbf{t}}$. So $\tau_c(\underline{a}ub)^R\in\mathcal{F}_{\mathbf{t}}$ which implies $\tau_c(u)\in \mathcal{F}_{\mathbf{t}}$. When $u=\sigma(w)$, the result follows from Lemma \ref{lem:sigmatau} and the inductive hypothesis in the same way.
\end{proof}

\begin{lemma}\label{lem:inequality}
Let $n\geq 1$ be an integer and $u\in\mathcal{F}_{\mathbf{t}}(n)$.
\begin{enumerate}
\item There exists $x\in\mathcal{F}_{\mathbf{t}}(\lfloor n/2\rfloor)$ such that 
\begin{equation}\label{eq:2-0}
 |x|_1-|x|_0-1 \leq  |u|_2-|u|_0 \leq |x|_1-|x|_0+1.
\end{equation}
\item There exists $y\in\mathcal{F}_{\mathbf{t}}(\lfloor n/2\rfloor)$ such that 
\begin{equation}\label{eq:1-0}
 |y|_1-|y|_2-1 \leq  |u|_1-|u|_0 \leq |y|_1-|y|_2+1.
\end{equation}
\item There exists $z\in\mathcal{F}_{\mathbf{t}}(\lfloor n/2\rfloor)$ such that 
\[ |z|_0-|z|_2-1 \leq  |u|_1-|u|_2 \leq |z|_0-|z|_2+1.\]
\end{enumerate}

\end{lemma}
\begin{proof}
(1) If $n$ is odd, then $u=a\sigma(v)$ or $\sigma(v)b$ where $v\in\mathcal{F}_{\mathbf{t}}(\lfloor n/2\rfloor)$ and $a,b\in\{0,1,2\}$. In either case, 
\begin{align*}
|u|_2-|u|_0 & =
\begin{cases}
|\sigma(v)|_2 - |\sigma(v)|_0 -1, & \text{ if }a,b=0,\\
|\sigma(v)|_2 - |\sigma(v)|_0, & \text{ if }a,b=1,\\
|\sigma(v)|_2 - |\sigma(v)|_0 +1, & \text{ if }a,b=2.
\end{cases} \notag \\
& =
\begin{cases}
|v|_1 - |v|_0 -1, & \text{ if }a,b=0,\\
|v|_1 - |v|_0, & \text{ if }a,b=1,\\
|v|_1 - |v|_0 +1, & \text{ if }a,b=2.
\end{cases} ~~~~~~~{(\text{by Lemma \ref{lem:image_num}})}
\end{align*}
Letting $x=v$, the result follows.

If $n$ is even, then $u=\sigma(w)$ or $a\sigma(v)b$ where $w\in\mathcal{F}_{\mathbf{t}}(n/2)$, $v\in\mathcal{F}_{\mathbf{t}}(\frac{n}{2}-1)$ and $a,b\in\{0,1,2\}$. When $u=\sigma(w)$, by Lemma \ref{lem:image_num}, $|u|_2-|u|_0 = |w|_1 - |w|_0$. Choosing $x=w$, we have the desired result. When $u=a\sigma(v)b$, let $\underline{a}=a-1\, (\mathrm{mod}~3)$ and $\overline{b}=b+1\, (\mathrm{mod}~3)$. Then $\underline{a}u\overline{b}=\sigma(\underline{a}vb)$. By Lemma \ref{lem:image_num}, 
\[|\underline{a}u\overline{b}|_2 - |\underline{a}u\overline{b}|_0
= |\sigma(\underline{a}vb)|_2-|\sigma(\underline{a}vb)|_0 
= |\underline{a}vb|_1 - |\underline{a}vb|_0,\]
which implies
\[|u|_2 - |u|_0
= |\underline{a}vb|_1 - |\underline{a}vb|_0+|\underline{a}\overline{b}|_0-|\underline{a}\overline{b}|_2.\]
When $ab\neq 00$ and $12$, 
\begin{equation*}
|u|_2-|u|_0 = |\underline{a}v|_1-|\underline{a}v|_0 + 
\begin{cases}
-1, & \text{ if } ab = 01, 20,\\
0,  & \text{ if } ab = 02, 10, 21,\\
1,  & \text{ if } ab = 11, 22.
\end{cases}
\end{equation*}
The result holds by choosing $x=\underline{a}v$. When $ab=00$ or $12$,
\begin{equation*}
|u|_2-|u|_0 = |vb|_1-|vb|_0 + 
\begin{cases}
-1, & \text{ if } ab = 00,\\
1,  & \text{ if } ab = 12.
\end{cases}
\end{equation*}
Setting $x=vb$, we are done.

(2) Let $u\in\mathcal{F}_{\mathbf{t}}$. By Lemma \ref{lem:fab}, $\tau_0(u)^R\in\mathcal{F}_{\mathbf{t}}$ and \[|u|_1-|u|_0=|\tau_0(u)^R|_2-|\tau_0(u)^R|_0.\] 
Applying \eqref{eq:2-0} to $\tau_0(u)^R$, we have $x\in\mathcal{F}_{\mathbf{t}}$ such that 
\[|x|_1-|x|_0-1 \leq  |\tau_0(u)^R|_2-|\tau_0(u)^R|_0 \leq |x|_1-|x|_0+1.\]
Let $y=\tau_1(x)^R$. Then, $y\in\mathcal{F}_{\mathbf{t}}$ and $|y|_1-|y|_2=|x|_1-|x|_0$. We have the desired result.

(3) Applying \eqref{eq:1-0} to $\tau_1(u)^R$ and letting $z=\tau_2(y)^R$, the result follows.
\end{proof}

Now we are ready to prove Proposition \ref{prop:dsubounds}.
\begin{proof}[Proof of Proposition \ref{prop:dsubounds}]
For every $n\geq 1$, there exists $k\geq 1$ such that $2^k\leq n < 2^k+1$. Suppose $u\in\mathcal{F}_{\mathbf{t}}(n)$. Let $n_1=n$. By Lemma \ref{lem:inequality}, we have $x^{(1)}\in\mathcal{F}_{\mathbf{t}}(\lfloor n_1/2\rfloor)$ such that
\[ |x^{(1)}|_1-|x^{(1)}|_0-1 \leq |u|_2-|u|_0 \leq |x^{(1)}|_1-|x^{(1)}|_0+1.\]
Let $n_2=\lfloor n_1/2\rfloor$. Apply Lemma \ref{lem:inequality} to $x^{(1)}$, we have
$x^{(2)}\in\mathcal{F}_{\mathbf{t}}(\lfloor n_2/2\rfloor)$ such that 
\[ |x^{(2)}|_1-|x^{(2)}|_2-1 \leq |x^{(1)}|_1-|x^{(1)}|_0 \leq |x^{(2)}|_1-|x^{(2)}|_2+1.\]
Therefore, 
\[ |x^{(2)}|_1-|x^{(2)}|_2-2 \leq |u|_2-|u|_0 \leq |x^{(2)}|_1-|x^{(2)}|_2+2.\]
Let $n_3=\lfloor n_2/2\rfloor$ and apply Lemma \ref{lem:inequality} to $x^{(2)}$. Then we have $x^{(3)}\in\mathcal{F}_{\mathbf{t}}(\lfloor n_3/2\rfloor)$ satisfying
\[ |x^{(3)}|_0-|x^{(3)}|_2-3 \leq |u|_2-|u|_0 \leq |x^{(3)}|_0-|x^{(3)}|_2+3.\]
After applying Lemma \ref{lem:inequality} $k$ times as above, we obtain that
\[-1-\lfloor \log_2 n\rfloor=-1-k \leq |u|_2-|u|_0 \leq 1+k=1+\lfloor \log_2 n\rfloor.\]
Hence \eqref{eq:ulbounds} holds.
\end{proof}

\subsection{Maximal and minimal digit sums}

Let $(d_k)_{k\geq -1}\in\{0,1,2\}^{\infty}$ where
\[ d_k=\begin{cases} 0 & \text{ if } k \equiv3,4 \mod 6, \\ 
1 & \text{ if } k \equiv 1,2 \mod 6, \\ 
2 & \text{ if } k \equiv0,5 \mod 6.
\end{cases} \]
Let $(c_\ell)_{\ell\geq 1}\in\{0,1,2\}^{\infty}$ given by $c_\ell=\ell +1\, (\mathrm{mod}~3)$. Applying Lemma \ref{lem:image_num} several times, it follows that for $k=0,1,2,3,4,5$, 
\begin{equation}\label{eq:d1c0}
|\sigma^{k}(d_k)|_2 - |\sigma^{k}(d_k)|_0 = 1\quad \text{and} \quad|\sigma^{k}(c_k)|_2 - |\sigma^{k}(c_k)|_0 = 0.
\end{equation}
In fact, these equalities hold for all $k\geq 1$.
\begin{lemma}\label{lem:5}
For all $\ell\geq 1$, 
\begin{equation*}
|\sigma^\ell(d_\ell)|_2 - |\sigma^\ell(d_\ell)|_0 = 1 \quad\text{and}\quad
|\sigma^\ell(c_\ell)|_2 - |\sigma^\ell(c_\ell)|_0 = 0.
\end{equation*}
\end{lemma}
\begin{proof}
Applying Lemma \ref{lem:image_num} six times, one obtain that for every $u\in\{0,1,2\}^{*}$, 
\begin{equation}\label{eq:six}
|\sigma^{6}(u)|_2-|\sigma^{6}(u)|_0=|u|_2-|u|_0.
\end{equation}
For all $\ell\geq 1$, we have $\ell = 6j+k$ where $j\geq 1$ and $k=0,1,2,3,4,5$. Then $d_\ell=d_k$ and $c_\ell=c_k$. By \eqref{eq:d1c0} and \eqref{eq:six},
\begin{align*}
|\sigma^{\ell}(d_\ell)|_2 - |\sigma^{\ell}(d_\ell)|_0 & =  |\sigma^{6j+k}(d_k)|_2 - |\sigma^{6j+k}(d_k)|_0
							= |\sigma^{k}(d_k)|_2 - |\sigma^{k}(d_k)|_0 = 1
\end{align*}
and 
\begin{align*}
|\sigma^{\ell}(c_\ell)|_2 - |\sigma^{\ell}(c_\ell)|_0 & =  |\sigma^{6j+k}(c_k)|_2 - |\sigma^{6j+k}(c_k)|_0 = |\sigma^{k}(c_k)|_2 - |\sigma^{k}(c_k)|_0 = 0.
\end{align*}
We have the desired.
\end{proof}

Now we define a sequence of words $\{W(n)\}_{n\geq 1}$ whose digit sums will attain the upper bound in Proposition \ref{prop:dsubounds}.  

Let $W_1:=2$. For $n \geq 2$, $W_n$ is defined as follows: suppose $2^k \leq n < 2^{k+1}$ for some $k$ and the 2-adic expansion of $n-2^k$ is written as \[(n-2^k)_2 = m_{k-1} \cdots m_2m_1m_0\]  where $m_{j}\in\{0,1\}$ for $j=0,1,\cdots k-1$. 
 Define 
\[W_L(n):=\delta_{m_0}2 \Big(\prod_{i=1}^{\lfloor \tfrac{k-1}{2} \rfloor }\sigma^{2i+m_{2i}}(d_{2i+m_{2i}})\Big) \]
and 
\[W_R(n):=\Big(\prod_{i=\lceil \tfrac{k-1}{2} \rceil}^{1}\sigma^{2i-1+m_{2i-1}}(d_{2i-1+m_{2i-1}}) \Big) 2\]
where $\delta_{m_0}=\varepsilon \text{ if }m_0=0$ and $1$ if $m_0=1$.
Let $W(n) := W_L(n)W_R(n).$

\begin{lemma}\label{lem:max_factor_subset}
For every integer $n$ satisfying $2^k \leq n < 2^{k+1}$ for some $k\geq 0$,  we have 
\begin{enumerate}
\item if $k$ is even, then $W_L(n)\triangleright \sigma^{k}(d_{k+1})$ and $W_R(n)\triangleleft\sigma^{k+1}(d_{k})$,
\item if $k$ is odd, then $W_L(n)\triangleright \sigma^{k+1}(d_{k+2})$ and $W_R(n)\triangleleft\sigma^k(d_{k-1})$.
\end{enumerate}
\end{lemma}

\begin{proof}
For $k=0,1,2$, the result can be verified directly from the definition of $W_L$ and $W_R$. Suppose the result hold for all $m\leq k$. Now we prove it for $m=k+1$. Let $2^{k+1}\leq n < 2^{k+2}$ with $(n-2^{k+1})_2=m_km_{k-1}\cdots m_1m_0$. Set $n^{\prime}=2^k+\sum_{i=0}^{k-1}m_{i}2^i$.

When $k+1$ is odd, write $k=2\ell$.  Then $W_R(n)=W_R(n^{\prime})\triangleleft\sigma^{k+1}(d_k)$ and 
\begin{align*}
W_L(n) & = \delta_{m_0}2 \Big(\prod_{i=1}^{\ell }\sigma^{2i+m_{2i}}(d_{2i+m_{2i}})\Big)\\
&  = \delta_{m_0}2 \Big(\prod_{i=1}^{\ell-1}\sigma^{2i+m_{2i}}(d_{2i+m_{2i}})\Big) \sigma^{2\ell+m_{2\ell}}(d_{2\ell+m_{2\ell}})\\
& = W_L(n^\prime)\sigma^{k+m_k}(d_{k+m_k}).
\end{align*}
When $m_k=0$, by the induction hypothesis, 
\begin{align*}
W_L(n^\prime)\sigma^{k+m_k}(d_{k+m_k})& \,\triangleright\, \sigma^{k}(d_{k+1})\sigma^{k}(d_k) =\sigma^{k}(d_{k+2})\sigma^{k}(d_k)\\
& = \sigma^{k+1}(d_{k+2}) \,\triangleright\, \sigma^{(k+1)+1}(d_{(k+2)+1}).
\end{align*}
When $m_k=1$, by the induction hypothesis,
\begin{align*}
W_L(n^\prime)\sigma^{k+m_k}(d_{k+m_k})& \,\triangleright\, \sigma^{k}(d_{k+1})\sigma^{k+1}(d_{k+1}) \,\triangleright\, \sigma^{k+1}(d_{k+3})\sigma^{k+1}(d_{k+1})\\
& = \sigma^{(k+1)+1}(d_{(k+1)+2}).
\end{align*}
So, $W_L(n)\triangleright\sigma^{(k+1)+1}(d_{(k+1)+2})$.

When $k+1$ is even, write $k=2\ell+1$. Then $W_L(n)=W_L(n^{\prime})\triangleright\sigma^{k+1}(d_{k+2})$ and 
\begin{align*}
W_R(n) & = \sigma^{k+m_k}(d_{k+m_k})W_R(n^{\prime}) \\
& \,\triangleleft\, \begin{cases}
\sigma^k(d_k)\sigma^k(d_{k-1})=\sigma^{k+1}(d_{k+1}), & \text{if }m_k=0,\\
\sigma^{k+1}(d_{k+1})\sigma^k(d_{k-1}), & \text{if }m_k=1,
\end{cases}\\
& \,\triangleleft\,\sigma^{k+2}(d_{k+1}).
\end{align*}
This completes the induction.
\end{proof}

\begin{proposition}\label{lem:max_f}
For all $n\geq 1$, $W(n)\in\mathcal{F}_{\mathbf{t}}$ and $\tau_1(W(n))^R\in\mathcal{F}_{\mathbf{t}}$. Moreover, 
\[\mathrm{DS}(W(n))=n+\lfloor\log_2n\rfloor+1\, \text{ and } \,\mathrm{DS}(\tau_1(W(n))^R)=n-\lfloor\log_2n\rfloor-1.\]
\end{proposition}

\begin{proof}
For all $n$ with $2^k \leq n < 2^{k+1}$, by Lemma \ref{lem:max_factor_subset}, 
\[W(n)=W_L(n)W_R(n)\prec \sigma^k(d_{k+2}d_kd_{k-2)} \prec\sigma^{k+1}(d_{k+2}d_{k-2}).\] Since $(d_\ell)_{\ell\geq 1}$ is periodic, $d_{k+2}d_{k-2}\in\{21,10,02\}\subset\mathcal{F}_\mathbf{t}(2)$. Hence $W(n)\in\mathcal{F}_{\mathbf{t}}$. By Lemma \ref{lem:fab}, we know $\tau_1(W(n))^R\in\mathcal{F}_{\mathbf{t}}$.  

According to the definition of $W_L$ and $W_R$,  
\begin{align*}
|W(n)|_2-|W(n)|_0 &= 2+ \sum_{i=1}^{k-1} (|\sigma^{i+m_i}(d_{i+m_i})|_2-|\sigma^{i+m_i}(d_{i+m_i})|_0)\\ 
&= 2+k-1 \qquad\qquad\text{(by Lemma \ref{lem:5})}\\
&= \lfloor \log_2(n)\rfloor+1.
\end{align*}
Then the results follow from \eqref{eq:dsn20} and the definition of $\tau_1$.
\end{proof}

\subsection{Accessible values of digit sums}
We shall prove the following intermediate value property of digit sums of all the factors of length $n$. 
\begin{proposition}\label{prop:3}
For all $n\geq 1$ and all integer $k$ satisfying $n-\lfloor \log_2 n\rfloor-1< k < n+\lfloor \log_2 n\rfloor+1$, there exists $u\in\mathcal{F}_{\mathbf{t}}(n)$ such that $\mathrm{DS}(u)=k$.
\end{proposition}

Before proving Proposition \ref{prop:3}, we first study the behavior of digit sums during the shift (to the right). Denote by $\mathcal{I}(u)$ the set of all the indexes (or positions) of occurrences of $u,$ i.e., for every $i \in  \mathcal{I}(u),$ $t_{i}t_{i+1}\cdots t_{i+n-1} = u.$ Since $\mathbf{t}$ is uniformly recurrent, $\mathcal{I}(u)$ is an infinite set for all $u\in \mathcal{F}_{\mathbf{t}}$. For every $i \in \mathcal{I}(u)$, set
\[ r_i(u) = \min\{j>i:g_n(j)>\mathrm{DS}(u)\}, \]
where $g_n(j):=\mathrm{DS}(t_j t_{j+1}\cdots t_{j+n-1})$.
Set $\min\emptyset=-\infty$. 
\iffalse
The quantity $r_i^{+}(u)$ (resp. $r_i^{-}(u)$) is the first increasing (resp. decreasing) position of digit sums for each occurrence of $u$ during the shift.
\fi

%For fixed $n$ and every word $\mathbf{u} \in \mathcal{F}_{\mathbf{t}}(n),$ it follows from Lemma \ref{lem:fab} and Lemma \ref{lem:max_f}, for every $i \in \mathcal{I}(u),$ either $r_i^{+}(\mathbf{u})$ or $r_i^{-}(\mathbf{u})$ is finite. Moreover, we have the following lemma.
\begin{lemma}\label{lem:mindiff}
Let $u\in\mathcal{F}_{\mathbf{t}}(n)$. If $\mathrm{DS}(u)\neq \mathrm{DS}(W(n))$, then $r_i(u)$ is finite and \[g_n(r_i(u))-\mathrm{DS}(u)=1\text{ or }2.\] Moreover, if $g_n(r_i(u))-\mathrm{DS}(u)=2$, then $g_n(r_i(u)-1)=\mathrm{DS}(u)$.
\end{lemma}
\begin{proof}
By Proposition \ref{prop:dsubounds} and \ref{lem:max_f}, if $\mathrm{DS}(u)\neq \mathrm{DS}(W(n))$, then $\mathrm{DS}(u)<\mathrm{DS}(W(n))$.  For any occurrence of $u$, say $t_i t_{i+1}\cdots t_{i+n-1}=u$, since $\mathbf{t}$ is uniformly recurrent, there exists $j>i$ such that $t_j t_{j+1}\cdots t_{j+n-1}=W(n)$. Thus, $r_i(u)<j$.

Now suppose $r_i(u)$ is finite. Write $k:=r_i(u)$. Then, $g_n(k-1)\leq \mathrm{DS}(u)$. Since \[g_n(k)-g_n(k-1)=t_{k+n-1}-t_{k-1}\in\{0,\pm 1,\pm 2\}\] and  $g_n(k)>\mathrm{DS}(u)\geq g_n(k-1)$, we know that  $g_n(k)-\mathrm{DS}(u)=1$ or $2$. 
Moreover, if $g_n(k-1)<\mathrm{DS}(u)$, then $0<g_n(k)-\mathrm{DS}(u)<g_n(k)-g_n(k-1)\leq 2$ which implies $g_n(k)-\mathrm{DS}(u)=1$.
\end{proof}

The key to prove Proposition \ref{prop:3} is to figure out how many times we need to do the shift in order to increase the digit sum of a given factor by $1$. The following two lemmas deal with the problem. The first one is a technical lemma. Let $u,v\in\mathcal{F}_{\mathbf{t}}(3)$. Write $\sigma^6(u)=u_0u_1\cdots u_{191}$ and $\sigma^6(v)=v_0v_1\cdots v_{191}$. For $64\leq i,j< 128$ satisfying $u_i=0$ and $v_j=2$, $0<m<192-\max\{i,j\}$ and $0<p\leq \min\{i,j\}$, set
\begin{align*}
R(u,v,i,j,m) & = \sum_{\ell=0}^{m}(v_{j+\ell}-u_{i+\ell}), \\
L(u,v,i,j,-p) & = \sum_{\ell=1}^{p}(u_{i-\ell}-v_{j-\ell}).
\end{align*}

\begin{lemma}\label{lem:tech}
For all $u,v\in\mathcal{F}_{\mathbf{t}}(3)$ and $64\leq i,j< 128$ satisfying $u_i=0$ and $v_j=2$, $R(u,v,i,j,m)=1$ for some $0<m<192-\max\{i,j\}$ or $L(u,v,i,j,-p)=1$ for some $0<p\leq \min\{i,j\}$.
\end{lemma}
\begin{proof}
Since the choices of variables of both $L$ and $R$ are finite, the result can be verified exhaustively. (This can be easily checked by a computer. We give the pseudocode for the corresponding procedures in Appendix A.) 
\end{proof}

\begin{lemma}\label{lem:bigodd}
Let $n> 128$. For every $u\in\mathcal{F}_{\mathbf{t}}(n)$ with $\mathrm{DS}(u)\neq \mathrm{DS}(W(n))$, there exists $z\in\mathcal{F}_{\mathbf{t}}(n)$ satisfying 
$\mathrm{DS}(z)-\mathrm{DS}(u)=1$.
\end{lemma}

\begin{proof}
Let $i\in\mathcal{I}(u)$ with $i\geq 2^8$. Set $j=r_i(u)-1$. By Lemma \ref{lem:mindiff}, if $g_n(r_i(u))-\mathrm{DS}(u)=1$, then we are done. If $g_n(r_i(u))-\mathrm{DS}(u)=2$, then $g_n(j)=\mathrm{DS}(u)$ which also implies $t_{j} = 0$ and $t_{j+n}=2$. Write $w=t_{j}t_{j}\cdots t_{j+n-1}$.

The word $w$ has the following decomposition:
\[ w = (t_{j}t_{j+1}\cdots t_{j+\ell-1})\sigma^{6}(v)(t_{j+n-r}\cdots t_{j+n-1}) \prec\sigma^6(xvy)\]
where $v \in \mathcal{F}_t$, $t_{j}t_{j+1}\cdots t_{j+\ell-1}\triangleright\sigma^{6}(x)$ and $t_{j+n-r}\cdots t_{j+n-1}\triangleleft\sigma^{6}(y)$ for some $x, y \in\{0,1,2\}$. Note that $\ell,r \leq 64.$ Further, we have \[w\prec\sigma^6(bxvyd)\] where $v \in \mathcal{F}_t,$ $bx,yd\in\mathcal{F}_{\mathbf{t}}(2)$ and $bxvyd\in\mathcal{F}_{\mathbf{t}}$. Let $\tilde{j}=j\, (\mathrm{mod}~ 64)$ and $j^\prime = j+n-1\, (\mathrm{mod}~ 64)$. Let $a,c\in\{0,1,2\}$ with $a\triangleleft v$ and $c\triangleright v$. Then,
\begin{align*}
g_n(j+m)-\mathrm{DS}(w) & = R(bxa,cyd,\tilde{j}+64,j^\prime+64,m),\\
g_n(j-p)-\mathrm{DS}(w) & = L(bxa,cyd,\tilde{j}+64,j^\prime+64,p).
\end{align*}
By Lemma \ref{lem:tech}, one of the following is true:
\begin{enumerate}
\item
$g_n(j+m)-\mathrm{DS}(w) = 1$ for some $0<m<192-\max\{i,j\}$;
\item
$g_n(j-p)-\mathrm{DS}(w) = 1$ for some $0<p\leq \min\{i,j\}$.
\end{enumerate}
Setting $z=t_{j+m}t_{j+m+1}\cdots t_{j+m+n-1}$ or 
$z=t_{j-p}t_{j-p+1}\cdots t_{j-p+n-1}$, we have the desired.
\end{proof}

Now we prove the intermediate value property of digit sums.
\begin{proof}[Proof of Proposition \ref{prop:3}]
For every $n>128$, starting from $\tau_1(W(n))^R$ and applying Lemma \ref{lem:bigodd} $2\lfloor \log_2n\rfloor+1$ times, the result follows. For $1\leq n\leq 128$, the result can be verified exhaustively. (This has been done by a computer. We provide the pseudocode for the corresponding procedure in Appendix B.)
\end{proof}

\begin{definition}
Let $\mathcal{A}\subset\mathbb{Z}$ be a finite set. An infinite sequence $\mathbf{u}=u_0u_1u_2\cdots\in\mathcal{A}^{\infty}$ is said to satisfy the \emph{intermediate value property} of digit sums if there exists an integer $N$ such that for every integer $n\geq N$ and every integer $i$ satisfying 
\[\min_{u \in \mathcal{F}_{\mathbf{u}}(n)} \mathrm{DS}(u)\leq  i \leq \max_{u \in \mathcal{F}_{\mathbf{u}}(n)}\mathrm{DS}(u) ,\] there exists a factor $v \in \mathcal{F}_{\mathbf{u}}(n)$ such that  $\mathrm{DS}(v) = i.$
\end{definition}
According to Proposition \ref{prop:3}, the Thue-Morse like sequence $\mathbf{t}$ has the intermediate value property of digit sums. It is easy to know that the intermediate value property holds for all $\mathbf{w} \in \{0,1\}^{\mathbb{N}}$. However, this property does not always hold. The following is a counter example.

Let the alphabet $\mathcal{A}=\{a,b,c\}$ and the morphism $\sigma_3: a\mapsto abc,~ b\mapsto bca,~ c\mapsto cab$. Consider the automatic sequence  $\mathbf{w}=\sigma_3^{\infty}(a)$ generated by $\sigma_3$. It follows from \cite{KK2017} that its abelian complexity $\{\rho_{\mathbf{w}}(n)\}_{n\geq 3}=766766766\cdots.$ 

\iffalse
Define two coding $f_1: a\mapsto 0, b \mapsto 1, c \mapsto 2$ and $f_2: a\mapsto 0, b \mapsto 1, c \mapsto 3.$ We have the following two claims which will be proved later.

\begin{claim}\label{clm:1}
For the infinite word $\mathbf{u} = f_1(\mathbf{w}),$ for every $n \geq 3,$ we have
\[ \{ \text{DS}(x): x\in  \mathcal{F}_{\mathbf{u}}(n)\} = \{n-2,n-1,n,n+1,n+2\}. \]
\end{claim}

\begin{claim}\label{clm:2}
For the infinite word $\mathbf{v} = f_2(\mathbf{w}),$ for every $m \geq 3,$ we have
\[ \{ \text{DS}(x): x\in  \mathcal{F}_{\mathbf{v}}(3m+1)\} = \{4m-2,4m,4m+1,4m+2,4m+3,4m+4\}. \]
\end{claim}
Claim \ref{clm:1} implies that $\mathbf{u}$ owns the intermediate value property. On the contrary, Claim \ref{clm:2} implies that $\mathbf{v}$ do not have the intermediate value property. 

Now we begin to prove the prescribed claims. 
\fi
Recall that $\Psi_{\mathbf{w}}(n)$ is the set of the Parikh vectors of all the factors of length $n$ of $\mathbf{w}$. Following from the proof of \cite[Proposition 4.1]{KK2017}, it is not hard to check the following proposition. 

\begin{proposition}\label{prop:4}
Let $\mathbf{w}=\sigma_3^{\infty}(a)$ with $\sigma_3: a\mapsto abc,~ b\mapsto bca,~ c\mapsto cab$ and set $\mathbb{I}=(1,1,1)$. For every integer $n \geq 3$ where $n=3m+r$ for some $r=0,1,2$, we have 
\begin{itemize}
\item if $r =0,$ then \[\Psi_{\mathbf{w}}(n)=m\mathbb{I}+
\{(1,0,-1),(0,0,0),(1,-1,0),(0,1,-1),(-1,1,0),(-1,0,1),(0,-1,1)\}. \]
\item if $r =1,$ then  \[\Psi_{\mathbf{w}}(n)=m\mathbb{I}+
\{(1,1,-1),(1,-1,1),(0,1,0),(1,0,0),(0,0,1),(-1,1,1)\}. \]
\item if $r =2,$ then \[\Psi_{\mathbf{w}}(n)=m\mathbb{I}+
\{(2,0,0),(0,2,0),(0,0,2),(0,1,1),(1,1,0),(1,0,1)\}. \]
\end{itemize}

\end{proposition}

\begin{example}
Given a coding $\tau:a\mapsto x, b\mapsto y, c\mapsto z$ where $x<y<z \in \mathbb{Z},$  consider the infinite word $\mathbf{w}^{\prime} = \tau(\mathbf{w})$ with $\mathbf{w}$ is defined in Proposition \ref{prop:4}. We will discuss the sum set of all the factors of length $n$ in $\mathbf{w}^{\prime}$ by the value of $r.$
\begin{itemize}
\item If $r=0,$ then we have
\begin{eqnarray*}
\{ \text{DS}(u): u\in  \mathcal{F}_{\mathbf{w}^{\prime}}(n)\} &=& \{<(x,y,z), \psi(n)>: \psi(n) \in \Psi_{\mathbf{w}}(n)\}\\
&=&x+y+z+\{x-z,y-z,x-y,0,y-x,z-y,z-x\}
\end{eqnarray*}
\item If $r=1,$ then we have
\begin{eqnarray*}
\{ \text{DS}(u): u\in  \mathcal{F}_{\mathbf{w}^{\prime}}(n)\} &=& \{<(x,y,z), \psi(n)>: \psi(n) \in \Psi_{\mathbf{w}}(n)\}\\
&=&x+y+z+\{x+y-z,y+x-z,x+z-y,x,y,z\}
\end{eqnarray*}
\item If $r=2,$ then we have
\begin{eqnarray*}
\{ \text{DS}(u): u\in  \mathcal{F}_{\mathbf{w}^{\prime}}(n)\} &=& \{<(x,y,z), \psi(n)>: \psi(n) \in \Psi_{\mathbf{w}}(n)\}\\
&=&x+y+z+\{x+y,y+x,x+z,2x,2y,2z\}
\end{eqnarray*}
\end{itemize}
Hence the intermediate value property of $\mathbf{w}^{\prime}$ holds if and only if $z=y+1=x+2.$
It follows that there are many infinite words such as  $\tau(\mathbf{w})$ with $\tau:a\mapsto 0, b\mapsto 1, c\mapsto 3$ in which the intermediate value property fails.
\end{example}
\iffalse
For $\mathbf{u} = f_1(\mathbf{w}),$ if $r=0,$ then 
\begin{eqnarray*}
\{ \text{DS}(x): x\in  \mathcal{F}_{\mathbf{u}}(n)\} &=& \{<(0,1,2), \psi(n)>: \psi(n) \in \Psi_{\mathbf{w}}(n)\}\\
&=&\{3m-2,3m-1,3m,3m+1,3m+2 \} \\
&=& \{n-2,n-1,n,n+1,n+2\}.
\end{eqnarray*}
The other cases follow from similar arguments. This gives Claim \ref{clm:1}.

For $\mathbf{v} = f_2(\mathbf{w}),$ it suffice to consider the case $r=1.$ In this case,
\begin{eqnarray*}
\{ \text{DS}(x): x\in  \mathcal{F}_{\mathbf{u}}(3m+1)\} &=& \{<(0,1,3), \psi(3m+1)>: \psi(3m+1) \in \Psi_{\mathbf{w}}(3m+1)\}\\
&=&\{4m-2,4m,4m,4m+1,4m+2,4m+3,4m+4\}.
\end{eqnarray*}
This completes the proof of Claim \ref{clm:2}.
\fi

As a consequence, we have the following interesting questions.
\begin{question}
Given a finite set $\mathcal{A}\subset\mathbb{Z}$, under what condition an infinite sequence $\mathbf{w}\in \mathcal{A}^{\mathbb{N}}$ will have the intermediate value property?
\end{question}
In addition, it follows from Proposition \ref{prop:3} and Lemma \ref{lem:image_num} that for the Thue-Morse like sequence $\mathbf{t},$  we have for every $n\geq 1,$
\[ \max_{a, b\in \{0,1,2\}}\max_{u \in \mathcal{F}_{\mathbf{t}}(n)}\{|u|_a-|u|_b\} = 2\lfloor \log_2 n\rfloor+3. \] 
On the basis of the above equation,  for an infinite word $\mathbf{w}$ on a finite alphabet $\mathcal{A}$, we can define a measure for the evenness degree of the factors of length $n$ in  $\mathbf{w}:$
\[ E_{\mathbf{w}}(n):=\max_{a, b\in \mathcal{A}}\max_{u \in \mathcal{F}_{\mathbf{w}}(n)}\{|u|_a-|u|_b\}. \]
Note that the evenness function $E_{\mathbf{w}}(n)$ is different with the balance function $B_{\mathbf{t}}(n)$ which is first introduced in \cite{A03}. For every primitive morphic word $\mathbf{w}$, Adamczewski \cite{A03} showed that the asymptotic behaviour of the balance function $B_{\mathbf{w}}(n)$  is in part ruled by the spectrum of the incidence matrix associated with the substitution. Here we have the following question.
\begin{question}
How about the asymptotic behaviour of the evenness function $E_{\mathbf{w}}(n)$ for the primitive morphic word $\mathbf{w}$?
\end{question}

\section*{Acknowledgement} 
This work was supported by  NSFC (No. 11701202), the Fundamental Research Funds for the Central Universities (Nos. 2017MS110, 2662015QD016) and the Characteristic innovation project of colleges and universities in Guangdong (No. 2016KTSCX007).

\section*{Appendix A: Pseudocode for Lemma \ref{lem:tech}.}
\makeatletter
\def\BState{\State\hskip-\ALG@thistlm}
\makeatother
\begin{algorithm}[H]
\begin{algorithmic}[1]
\Procedure{RightShiftTimes}{}
\BState \textbf{Input:} $u,v,i,j$ with $64\leq i,j< 128$ satisfying $u_i=0$ and $v_j=2$
\BState \textbf{Output:} $m$ or false
\State $lword \gets \sigma^6(u)$
\State $rword \gets \sigma^6(v)$
\State $m \gets 0$
\State $s \gets 0$
\While{$m \leq 192-\max(i, j)$} 
\State $s \gets s + rword(j+m)-lword(i+m)$
\If {$s = 1$}
\State \Return $m$
\EndIf
\State $m \gets m+1$
\EndWhile
\State \Return false
\EndProcedure

\Procedure{LeftShiftTimes}{}
\BState \textbf{Input:} $u,v,i,j$ with $64\leq i,j< 128$ satisfying $u_i=0$ and $v_j=2$
\BState \textbf{Output:} $-p$ or false
\State $lword \gets \sigma^6(u)$
\State $rword \gets \sigma^6(v)$
\State $p \gets 1$
\State $s \gets 0$
\While{$p \leq \min(i, j)$}
\State $s \gets s + lword(i-p)-rword(j-p)$
\If {$s = 1$}
\State \Return $-p$
\EndIf
\State $p \gets p+1$
\EndWhile
\State \Return false
\EndProcedure
\end{algorithmic}
\caption{For every input $u,v,i,j$ which is present in Lemma \ref{lem:tech}, the outputs of two following procedures can not be both false.}
\end{algorithm}

\section*{Appendix B: Pseudocode for Proposition \ref{prop:3} for $1\leq n\leq 128$.}

Since $\mathbf{t}$ is uniformly recurrent, for every positive integer $n$, there exits an integer $R(n)>n$  such that for every $u \in \mathcal{F}_{\mathbf{t}}(n)$,  we have $u \prec t_0\cdots t_{R(n)}.$ At the mean time, using the analogue of the proof of \cite[Proposition 5.1.9]{F02}, we can show the subword complexity function $\rho_{\mathbf{t}}(n)$: $\rho_{\mathbf{t}}(1)=3, \rho_{\mathbf{t}}(2)=9,$ and for $n \geq 3,$
\[\left\{
\begin{aligned}
\rho_{\mathbf{t}}(2n) &= \rho_{\mathbf{t}}(n)+\rho_{\mathbf{t}}(n+1),\\
\rho_{\mathbf{t}}(2n+1) &= 2\rho_{\mathbf{t}}(n+1).
\end{aligned} \right.\]
Hence it is possible to find the index $R(n)$ for every $1\leq n\leq 128$ with the help of a computer.
\begin{algorithm}[H]
\begin{algorithmic}[1]
\Procedure{HaveDesiredDigitSum}{}
\BState \textbf{Input:} $n,k$ with $1\leq n\leq 128$ and $n-\lfloor \log_2 n \rfloor-1 < k < n+\lfloor \log_2 n \rfloor+1$
\BState \textbf{Output:} true or false
\State $i \gets 0$
\While{$i \leq R(n)-n+1$} 
\State $ds \gets 0$
\State $j \gets i$
\While{$j\leq i+n-1$} 
\State $ds \gets ds+t_j$
\State $j \gets j+1$
\EndWhile
\If {$ds = k$}
\State \Return true
\EndIf
\State $i \gets i+1$
\EndWhile
\State \Return false
\EndProcedure
\end{algorithmic}
\caption{For every input $n,k$, the output of the following procedure always be true.}
\end{algorithm}

\end{document}